%% file: draft.tex
\documentclass[twoside,12pt]{scrbook}
\usepackage{amsmath,amssymb,amsthm}
\usepackage{enumerate}
\usepackage{cite}
\usepackage{hyperref}
\usepackage{dsfont}

\include{thesisdef}

\usepackage[nottoc,numbib]{tocbibind}
\usepackage[all]{xy}
\usepackage[onehalfspacing]{setspace}
\usepackage{mathrsfs}

\newtheorem{chapLemma}{Lemma}[chapter]
\newtheorem{secLemma}{Lemma}[section]

\newtheorem*{Lemma*}{Lemma}
\newtheorem{chapTheorem}[chapLemma]{Theorem}

\newtheorem*{theorem*}{Theorem}
\newtheorem{chapProp}[chapLemma]{Proposition}
\newtheorem{secProp}[secLemma]{Proposition}

\newtheorem*{Prop*}{Proposition}
\newtheorem{chapCor}[chapLemma]{Corollary}

\newtheorem*{Cor*}{Corollary}

\theoremstyle{definition}
\newtheorem{chapDef}[chapLemma]{Definition}
\newtheorem{secDef}[secLemma]{Definition}

\newtheorem*{Def*}{Definition}

\newtheorem*{Obs*}{Observation}

\newtheorem*{Fact*}{Fact}
\newtheorem{Ex}{Example}
\newtheorem{Qn}{Question}

\newtheorem{chapRemark}[chapLemma]{Remark}
\newtheorem{secRemark}[secLemma]{Remark}

\newtheorem*{Remark*}{Remark}

\newtheorem*{Notation*}{Notation}

\title{DILATIONS OF MATRICES}
\author{David Cohen\\	
		Advisor: Dr. Orr Moshe Shalit\\}
\date{June 2014}
\publishers{Ben-Gurion University of the Negev\\The Faculty of Natural Sciences\\Department of Mathematics}

\begin{document}
\maketitle

\chapter*{Abstract}
We explore aspects of dilation theory in the finite dimensional case 
and show that for a commuting $n$-tuple of operators $T=(T_1,...,T_n) $ acting 
on some finite dimensional Hilbert space $H$ and a compact set 
$X\subset \mathbb{C}^n$ the following are equivalent:
\begin{enumerate}
\item 
$T$ has a normal $ X$-dilation.
\item
For any $m\in \mathbb{N}$ there exists some 
finite dimensional Hilbert space $K$ containing $H$ and
 a tuple of commuting normal operators $N=(N_1,...,N_n)$ acting on $K$ 
such that $$  q(T)=P_Hq(N)|_H$$ 
for all polynomials $q$ of degree at most $m$
and such that the joint spectrum of $N$ is contained in $X$ (where $P_H$ is the projection
from $K$ to $H$).
\end{enumerate}
While $(2) \Rightarrow (1)$ is a simple consequence of Arveson's dilation theorem,
in order to show the other direction we prove    
that for a certain type of positive operator valued measure (POVM) and for a finite dimensional subspace $V$ of $C(X)$ there exists a cubature formula of the POVM with respect to the functions in $V$.
We then, by using Naimark's dilation theorem, obtain a spectral measure from which we are able to construct the wanted dilation. 
\tableofcontents
\chapter{Introduction}

We shall begin with a short survey on dilation theory.

\begin{chapDef}

Let $T$ be an operator acting on a Hilbert space $H$. 
A \textbf{dilation} of $T$
is an operator $A$ acting on a Hilbert space $K$ containing $H$ such that
$$ T^m=P_HA^m | _H \text{,}$$
for all $m\in \mathbb{N}$,
where $P_H$ denotes the projection of $K$ onto $H$.
\end{chapDef}

One of the most 
important results concerning dilations was due to Sz.-Nagy:
\begin{chapTheorem} [{\textbf{Sz.-Nagy dilation theorem} \cite{nagy1953contractions}  }]
Let $T$ be a
contraction acting on a Hilbert space $H$. Then there exists a Hilbert space
$K$ containing $H$ and a unitary operator $U$ acting on $K$ such that:
$$ T^m=P_HU^m | _H \ ,\ m\in \mathbb{N}.$$ 

\end{chapTheorem}

There are quite a few uses of Sz.-Nagy's theorem let us just mention the Von Neumann inequality.
\begin{chapTheorem}  [{\textbf{von Neumann's inequality}  }]
Let $T$ be a contraction  acting on a Hilbert space $H$.
Then for any polynomial $p$ in $\mathbb{C}[z]$ we have that: 
$$\Vert p(T) \Vert \leq \Vert p(z) \Vert_{\infty,\mathbb{D}}.$$
\end{chapTheorem}

Though the original proof of von Neumann’s inequality predates Sz.-Nagy's theorem, the von Neumann’s  inequality can be derived from Sz.-Nagy's theorem by a simple use of the continuous functional calculus for normal operators.

We can also talk about the multi-variable case.
\begin{chapDef}
Let $T=(T_1,...,T_n)$ be a commuting tuple of operators acting on a Hilbert space $H$. 
A dilation of $T$
is a tuple of commuting operators $A=(A_1,...,A_n)$ acting on a Hilbert space $K$ containing $H$ such that
$$T_1^{m_1}...T_n^{m_n}=P_HA_1^{m_1}...A_n^{m_n} \upharpoonright _H ,$$
for all $m_1,...m_n\in \mathbb{N}$
\end{chapDef}

Ando's theorem \cite{ando1963pair} 
asserts that every pair of commuting contractions has a unitary dilation, i.e. , a dilation in which both operators 
are unitaries. This can be viewed as an analogue to Sz.-Nagy's theorem when $n=2$, and we again obtain the following von Neumann type inequality for two commuting contractions. 

\begin{chapTheorem}
Let $T_1,T_2$ be a pair of commuting contractions acting on a Hilbert space $H$ then
$$\Vert p(T_1,T_2)\Vert \leq \Vert p(z_1,z_2) \Vert _{\infty , \mathbb{D} ^2 }, $$
far all $p\in \mathbb{C}[z_1,z_2]$.
\end{chapTheorem}
Varopoulos has shown in \cite{varopoulos1974inequality}
that for some positive integer $n$ greater then two 
there exists a commuting $n$-tuple of contractions for which the 
von Neumann inequality fails, and in particular, a unitary dilation can not exist.
He was also able to provide an explicit example of three commuting $5\times 5$
contractions for which this occurs. 
Following this Crabb and Davie gave another example of three $8\times 8 $ commuting contractions for which the von Neumann inequality does not hold (see \cite{crabb1975neumann}) and Holbrook was even able to find an example of three $4\times 4$ commuting contractions \cite{holbrook2001schur}. 
In \cite{parrott1970unitary} Parrott shows how to construct an example of three commuting contractions for which Von Neumann's inequality holds, but still a unitary dilation
does not exists. It is still unknown if there exists three $3\times 3$ commuting contractions for which the inequality fails.
 
Let us now introduce another type of dilation.

\begin{chapDef}
Let $X$ be a compact subset on the complex plane and let $\mathcal{R}(X)$ denote the algebra of all rational functions with poles off $X$. 
Let $T$ be an operator in $B(H)$. A \textbf{normal $\partial X$ dilation} of an operator $T\in B(H)$ will consist of of a normal
operator $N$ acting on some Hilbert space $K$ containing $H$ such 
that the spectrum of $N$ is contained in the (topological) boundary of $X$ and such
that for any $r\in \mathcal{R}(X)$
$$ r(T)=P_Hr(N) | _H .$$  

\end{chapDef}

A notion that is closely related to normal $\partial X$ dilations is the one of spectral sets.
\begin{chapDef}
Let $X\subset \mathbb{C}$ be compact.
$X$ will be called a \textbf{spectral set} of $T\in B(H)$ if 
the spectrum of $T$ is contained in $X$ and
for every $f\in \mathcal{R}(X)$
$$\Vert f(T) \Vert \leq \Vert f(z) \Vert_{\infty,X}.$$
$X$ will be called a \textbf{complete spectral set} for $T$ if it is a spectral set for $T$ and if for any $l\in \mathbb{N}$ and any $l\times l$ matrix with 
entries in $\mathcal{R}(X)$ one has that 
$$\Vert (f_{i,j}(T)) \Vert \leq \sup_{z\in X} \Vert (f_{i,j}(z))\Vert,$$
where the norm on the left side of the inequality is the operator norm on the direct sum  $\oplus_{i=1} ^l H$.
\end{chapDef}

Arveson's dilation theorem \cite[Theorem 1.2.2.]{arveson1972subalgebras} shows that for an operator $T\in B(H)$ having a normal $\partial X$ dilation is equivalent to having $X$ as a complete spectral set for $T$. It is known that for certain cases it is enough for 
$X$ to be a spectral set for $T$ in order to guarantee the existence
of a normal $\partial X$ dilation \cite[p. 48, Theorem 4.4]{paulsen2002completely}. For example, if the interior of a compact set $X$ is simply connected and $\mathbb{C} \setminus X$ has only finitely many components then whenever $X$ is a spectral set for some operator $T$, a normal $\partial X$ dilation exists. Another example of a set with this property is the annulus as 
was shown be Agler in \cite{agler1985rational}. 

On the other hand it was shown by Dritschel and McCullough as well as by Agler, Harland and Raphael that for a bounded triply connected domain $X$ with boundary consisting of disjoint analytic curves there exists an operator for which $X$ is a spectral set
but does not have a normal $\partial X$ dilation. Agler, Harland and Raphael even gave an example of a $4\times 4$ matrix for which this occurs (see \cite{dritschel2005failure, agler2008classical} for details).

One of our motivations was to get a better understanding of these  phenomena through the finite dimensional case.
We now restrict ourselves to the the case where $H$ is a finite dimensional
Hilbert space and $T=(T_1,...,T_n)$ is a commuting tuple of contractions
in $B(H)$.

\begin{chapDef}
An $n$-tuple of contractions $T$ in $B(H)$ is said to have a unitary $m$-dilation if there exist a Hilbert space $K$ containing $H$ and an $n$-tuple of commuting unitaries $U=(U_1,...U_n)$ acting on $K$ such that $T_1^{m_1}...T_n^{m_k}=P_HU_1^{m_1}...U_n^{m_k}P_H$ for all 
non-negative integers $m_1,...m_k$ satisfying $m_1+...+m_k \leq m$.

\end{chapDef}

\begin{Ex}
Let $T$ be a contraction in $B(H)$. Then the operator 
$$\begin{bmatrix}
T&(I-TT^*)^{\frac{1}{2}}
\\(I-T^*T)^{\frac{1}{2}} &-T^*
\end{bmatrix}$$
 is a unitary 1-dilation of $T$ acting on $H\oplus H$.
\end{Ex}
It is worth mentioning that there are some uses to $m$-dilations such as given in
\cite[Theorem 4.7]{levy2010dilation}, which present a finite dimensional variant of the von Neumann inequality. In \cite{shalit2013sneaky} these dilations are used in order to provide a proof of the maximum modulus principle (in the unit disc) that it is based on linear algebra.

In \cite{mccarthy2013unitary} McCarthy and Shalit proved the following theorem:
\begin{chapTheorem}
\label{SHA} 
Let $H$ be a finite dimensional Hilbert space. Let $T_1,...,T_n$ be commuting contractions on $H$. The following are equivalent.
\begin{enumerate}
\item
The $k$-tuple $T_1,...,T_n$ has a unitary dilation.
\item
For every $m\in \mathbb{N}$, the $n$-tuple $T_1,...,T_n$ has a unitary $m$-dilation which acts on a finite dimensional space.
\end{enumerate}
\end{chapTheorem}

One of ours objectives was to generalize this theorem
to the case where the dilation will consist of normal operators with limitation on
the set which contains their spectrum.
The case of a unitary $m$-dilation can be considered as a normal dilation in which we require that each one of the operators in the dilation has its spectrum contained inside $\mathbb{T}.$  

Let us now state the main result of this thesis.

\begin{chapTheorem}
\label{I1}
Let $H$ be a Hilbert space of dimension $d$ and let $X\subset \mathbb{C}^n$ compact. Let $T_1,...,T_n$ be commuting operators on $H$. The following are equivalent.
\begin{enumerate}
\item
The $n$-tuple $T_1,...,T_n$ has a normal polynomial $X$ dilation (see \ref{PD} for the definition).
\item
For any $m\in \mathbb{N}$ there exists some 
finite dimensional Hilbert space $K$ containing $H$ and
 a normal $n$-tuple of commuting operators acting on $K$ 
 such that the joint spectrum of $N$ is contained in  $X$ and  
such that $$  q(T)=P_Hq(N)|_H$$ 
for all polynomials q of degree at most $m$.
\end{enumerate}
\end{chapTheorem}
In their proof McCarthy and Shalit make use of the 
 Poisson kernel on the polydisk, which is not at our disposal in the more general setting.
 In order to obtain our result
  we needed to find appropriate tools which will serve the role taken by the Poisson kernel, but this time for an arbitrary compact sets of $\mathbb{C}^n$.
 Doing so we were able to show that given a certain type of positive operator valued measure
and some finite set of functions which are continuous on the support of the measure we can find a cubature formula. We shall discuss this in greater detail as we proceed.

\begin{chapRemark}
If, in Theorem \ref{I1}, we set $X = \mathbb{T}^n$, then we recover Theorem \ref{SHA}.
\end{chapRemark}

\chapter{Preliminaries}
\section{The joint spectrum}
It is a well known theorem that for a bounded operator $B$ acting on a Banach space $X$ one can construct a holomorphic functional calculus, namely,
if we denote by $H(\sigma(B))$ the set of all
functions which are holomorphic in some neighborhood of the spectrum of $B$ then
the map $F \rightarrow F(T)$ from $H(\sigma(B))$ to $L(X)$ 
has the following  properties:
\begin{itemize}
\item
It extends the polynomial functional calculus.
\item
It is an algebra homomorphism from the algebra of holomorphic functions defined on a neighborhood of $\sigma(B)$ to $L(X)$
\item
It preserves uniform convergence on compact sets.
\end{itemize}

Given a tuple of commuting bounded operators acting on some Banach space
there are several analogues of what we can consider to be the ``spectrum''
of such tuple in order to construct a multi-variable holomorphic functional calculus.
In this section we shall briefly present two different ways to go about doing this and focus on some cases which are relevant to our discussion.
Let us begin with some notations.

\begin{chapDef}
Let $V$ be an open set in $\mathbb{C}^d$. We will say that a function $F: V \rightarrow \mathbb{C}$ is \textbf{holomorphic} on $V$ if it is locally expandable as a convergent power series in the variables $z_1,...,z_d$.
The algebra (with respect to point-wise addition and multiplication) of all such functions will be denoted be $H(V)$.
\end{chapDef}   
 
\begin{chapDef}
Let $K$ be a compact set in $\mathbb{C}^d$. A function $F$ will said to be \textbf{holomorphic on $K$} if there exists some open neighborhood $U$ of $K$ such that
$F\in H(U)$. The set of all such functions will be denoted by $H(K)$ thus $H(K)= \bigcup _U H(U)$ where $U$ runs over all the open neighborhoods of $K$. 

\end{chapDef}

We will present two different definitions for the ``spectrum of a commuting tuple
of operators''.
 In both cases we will call such a set the \textbf{joint spectrum} of the corresponding tuple 
and both will yield us a functional calculus with
the same properties mentioned before.
 We now turn to present the first approach:

\begin{chapDef}
Given a unital commutative Banch algebra $\textbf{A}$ and 
$A=(A_1,...,A_n) \in \textbf{A}^n $ we say $A$ is invertible if there exists $B=(B_1,...,B_n) \in \textbf{A}^n $ such that $$ \sum _i A_i B_i = 1_{ \textbf{A} }.$$
The \textbf{algebraic spectrum} of $A$ is defined to be as follows:
$$ \sigma_{\textbf{A}}(A)= \{ \lambda \in \mathbb{C}^n : A-\lambda= (A_1-\lambda_1 1_{ \textbf{A} },..., A_n - \lambda_n 1_{ \textbf{A} } ) \text{ is not invertible in \textbf{A}} \} $$

\end{chapDef}

It can be shown without great difficulty that 
$$ \sigma_{\textbf{A}}(A)= \{ (m(A_1),...,m(A_n)) : m \text{ is in the maximal ideal space of \textit{\textbf{A}} } \}.$$ 

One should take notice that this definition is dependent on a commutative algebra that contains the mentioned operators. 
In our context  $A_1,...A_n$ will always be bounded operators acting on some Hilbert space $H$ and 
all the algebras will be unital.
We also take notice that if $\textbf{A},\textbf{B}$
are two commutative unital sub-algebras of $B(H)$ containing $A_1,...,A_n$
and such that $\textbf{A} \subset \textbf{B}$ then $\sigma_{\textbf{B}}(A) \subset \sigma_{\textbf{A}}(A)$. Let us from now on denote by $\textbf{A}$ the commutative algebra generated
by $A_1,...,A_n$ in $B(H)$. Since any other commutative algebra
containing $A_1,...,A_n$ contains $\textbf{A}$ we have that $\sigma_{\textbf{A}}(A)$
is the maximal set (with respect to inclusion) of all such algebraic joint spectra of $A$. \\

The other type of spectrum we present  is called the \textbf{Taylor joint spectrum }.
The definition of this spectrum is considerably more elaborate, but has the advantage of being independent on the underlying algebra.

We start with the following notation.

Let $\Lambda$ be the exterior algebra on $n$ generators $e_1,...e_n$, with identity 
$e_0\equiv 1$. $\Lambda$ is the algebra of forms in $e_1,...e_n$ with complex coefficients, subject to the collapsing property $e_ie_j+e_je_i=0$ ($1\leq i,j \leq n$). Let $E_i: \Lambda \rightarrow \Lambda$ denote the creation operator, given by
$E_i \xi = e_i \xi $ ($\xi \in \Lambda, 1 \leq i \leq n$).
 If we declare $ \{ e_{i_1}... e_{i_k} : 1 \leq i_1 < ... < i_k \leq n \}$ to be an 
 orthonormal basis, the exterior algebra $\Lambda$ becomes a Hilbert space, admitting an orthogonal decomposition $\Lambda = \oplus_{k=1} ^n \Lambda^k$ where $\dim \Lambda ^k = {n \choose k}$. Thus, each $\xi \in \Lambda$ admits a unique orthogonal decomposition
 $ \xi = e_i \xi' + \xi''$, where $\xi'$ and $\xi ''$ have no $e_i$
contribution. It then follows that that  $ E_i ^{*} \xi = \xi' $, 
and we have that each $E_i$ is a partial isometry, satisfying $E_i^*E_j+E_jE_i^*=\delta_{i,j}$.
Let $X$ be a normed space, let $A=(A_1,...,A_n)$ be a commuting $n$-tuple of bounded
operators on $X$ and set $\Lambda(X)=X\otimes_{\mathbb{C}} \Lambda$.
We define $D_A: \Lambda (X) \rightarrow \Lambda (X)$ by 

$$D_A = \sum_{i=1}^n A_i \otimes E_i .$$ Then it is easy to see $D_A^2=0$, so
$\Ran D_A \subset \Ker D_A$.
The commuting $n$-tuple is said to be \textbf{non-singular} on $X$ if $\Ran D_A=\Ker D_A$.
\begin{chapDef}
The Taylor joint spectrum of $A$ on $X$ is the set 
$$\sigma_T(A,X) = \{ \lambda\in \mathbb{C}^n : A-\lambda \text{ is singular} \}.$$
\end{chapDef}
\begin{chapRemark}
The decomposition $\Lambda=\oplus_{k=1}^n \Lambda^k$
gives rise to a cochain complex $K(A,X)$, known as the Koszul complex associated to $A$ on $X$, as follows:
$$K(A,X):0 \rightarrow \Lambda^0(X)\xrightarrow{D_A^0}... \xrightarrow{D_A^{n-1}} \Lambda^n(X) \rightarrow 0 ,$$
where $D_A^{k}$ denotes the restriction of $D_A$ to the subspace $\Lambda^k(X)$. Thus,
$$ \sigma_T(A,X) = \{ \lambda\in \mathbb{C}^n : K(A-\lambda ,X)\text{ is not exact} \} .$$
\end{chapRemark}

Another property of the Taylor spectrum is that it is the smallest spectrum in the 
 following sense:
  
\begin{chapTheorem}
\label{ABC}
 
Let $A=(A_1,...,A_n)$ be a commuting tuple in $B(H)$ and let $\textbf{B}$ some commutative
 sub-algebra containing them. Then $\sigma_T(A) \subset \sigma_{\textbf{B}}(A)$.
 \end{chapTheorem} 
See \cite[Lemma 1.1]{taylor1970joint} for more Details.
\\

We now turn to focus on two specific cases.
\\
\\
\subsection{The finite dimensional case}

We remind the reader of a known result in linear algebra 
( \cite{halmos1947finite} Sec. 56 Thm. 1).
If $A_1,...,A_n$ is an $n$-tuple of commuting $d\times d$ matrices over the complex field, then there exists a unitary matrix $U$ such that for every $1\leq j \leq n$, $U^*A_jU$ is upper triangular. Thus,

$$U^*A_jU=\begin{bmatrix}
\lambda_1^{(j)}&*&*&...&*\\ 0& \lambda_2^{(j)}&*&...&*
\\ 0&0& \lambda_3^{(j)} &...&*
\\ \vdots & \vdots & \vdots & \ddots & \vdots
\\ 0 & 0 & 0 & 0 & \lambda_d^{(j)}
\end{bmatrix} \text{ , j=1,...n,}  $$
where $\lbrace \lambda_1^{(j)},...,\lambda_d^{(j)} \rbrace$ is the spectrum 
of $A_j$ (perhaps with repetitions).

\begin{secProp}
\label{MM}
Let $A=(A_1,...,A_n)$ be commuting $d \times d$ matrices over the complex field, $\textbf{A}$ be the commutative unital Banach algebra generated
by $A$ and let $M=\lbrace (\lambda_k^{(1)},\lambda_k^{(2)},...,\lambda_k^{(n)}) 
;$ $ k=1,...,d \rbrace \subset \mathbb{C}^n$.
Then 
$\sigma_\textbf{A}(A)$ 
 is equal to $M$.
\end{secProp}
\begin{proof}
We begin by showing $M \subset \sigma_\textbf{A}(A)$.
Let $\lambda_k=(\lambda_k^{(1)},\lambda_k^{(2)},...,\lambda_k^{(n)})$
be some point in $M$.
 We want to show that for every 
$C_1,...,C_n$ in $\textbf{A}$, 
$\sum_jC_j(A_j-\lambda_k^{(j)})$ is not the identity.
Since the invertible elements form an open set and the polynomials in 
$A_1,...,A_n$ are dense in $\textbf{A}$ it will suffice to show that for any
$p_1,...,p_n\in \mathbb{C}[z_1,...,z_n]$, the operator
 $\sum_jp_j(A)(A_j-\lambda_k^{(j)})$
is not invertible in $\textbf{A}$.
Now let $p_1,...,p_n\in \mathbb{C}[z_1,...,z_n]$.
Note that for every $p \in \mathbb{C}[z_1,...,z_n]$ we have:

$$p(A)=U\begin{bmatrix}

p(\lambda_1^{(1)},...,\lambda_1^{(n)})&*&*&...&*\\ 0& 
p(\lambda_2^{(1)},...,\lambda_2^{(n)})&*&...&*
\\ 0&0& 
p(\lambda_3^{(1)},...,\lambda_3^{(n)}) &...&*
\\ \vdots & \vdots & \vdots & \ddots & \vdots
\\ 0 & 0 & 0 & 0 & 
p(\lambda_d^{(1)},...,\lambda_d^{(n)})
\end{bmatrix} U^*.$$
Thus if we set $q(z_1,...,z_n)=\Sigma_jp_j(z_1,...,z_n)(z_j-\lambda_k^{(j)})$,
then $\Sigma_jp_j(A)(A_j-\lambda_k^{(j)})= 
 q(A)=$ $$U\begin{bmatrix}

q(\lambda_1^{(1)},...,\lambda_1^{(n)})&*&*&...&*\\ 0& 
q(\lambda_2^{(1)},...,\lambda_2^{(n)})&*&...&*
\\ 0&0& 
q(\lambda_3^{(1)},...,\lambda_3^{(n)}) &...&*
\\ \vdots & \vdots & \vdots & \ddots & \vdots
\\ 0 & 0 & 0 & 0 & 
q(\lambda_d^{(1)},...,\lambda_d^{(n)})
\end{bmatrix}U^*.$$
Since $q(\lambda_k^{(1)},...,\lambda_k^{(n)})=0$ we get that $0$ is an eigenvalue for $\Sigma_jp_j(A)(A_j-\lambda_k^{(j)})$, that is, $\Sigma_jp_j(A)(A_j-\lambda_k^{(j)})$
is not invertible. This shows $\lambda_k$ is in $\sigma_{\textbf{A}}(A)$.

In the other direction let $\alpha =(\alpha_1,...,\alpha_n)\in \mathbb{C}^n$
and assume that for every $C_1,...,C_n $ in $ \textbf{A}$ we have that
$\Sigma_j C_j(A_j-\alpha_j)$ is not invertible. We will show $\alpha=\lambda_k$ for
some $k$.

Now for every $j=1,...,n$ we have 
$$A_j-\alpha_jI=U\begin{bmatrix}

\lambda_1^{(j)}-\alpha_j&*&*&...&*\\ 0& 
\lambda_2^{(j)}-\alpha_j&*&...&*
\\ 0&0& 
\lambda_3^{(j)}-\alpha_j &...&*
\\ \vdots & \vdots & \vdots & \ddots & \vdots
\\ 0 & 0 & 0 & 0 & 
\lambda_d^{(j)}-\alpha_j
\end{bmatrix}U^*.$$
Let $q_j \in \mathbb{C}[z_j]$ be such that
$q_j(\lambda_k^{(j)})=\overline{\lambda_k^{(j)} - \alpha_j}$ for every $k=1,...,d$ 
(one can take for example a suitable Lagrange interpolation polynomial).
Then $$q_j(A_j)(A_j-\alpha_jI)=
U\begin{bmatrix}

\vert \lambda_1^{(j)}-\alpha_j \vert ^2 &*&*&...&*\\ 0& 
\vert \lambda_2^{(j)}-\alpha_j \vert^2 &*&...&*
\\ 0&0& 
\vert \lambda_3^{(j)}-\alpha_j \vert^2 &...&*
\\ \vdots & \vdots & \vdots & \ddots & \vdots
\\ 0 & 0 & 0 & 0 & 
\vert \lambda_d^{(j)}-\alpha_j \vert^2
\end{bmatrix}U^*,$$

and so we get

 $$\Sigma_j q_j(A_j)(A_j-\alpha_j)= 
U\begin{bmatrix}

\Sigma_j \vert \lambda_1^{(j)}-\alpha_j \vert ^2&*&*&...&*\\ 0& 
\Sigma_j \vert \lambda_2^{(j)}-\alpha_j \vert ^2 &*&...&*
\\ 0&0& 
\Sigma_j \vert \lambda_3^{(j)}-\alpha_j \vert  ^2 &...&*
\\ \vdots & \vdots & \vdots & \ddots & \vdots
\\ 0 & 0 & 0 & 0 & 
\Sigma_j \vert \lambda_d^{(j)}-\alpha_j \vert ^2
\end{bmatrix}U^*
.$$

By assumption $\Sigma_j q_j(A_j)(A_j-\alpha_j)$ is singular thus we get that for some $1\leq k \leq d$ we have $\Sigma_j \vert \lambda_{k}^{(j)}-\alpha_j \vert ^2=0$ which implies $\alpha_j=\lambda_k^{(j)}$ for $j=1,..,n$ and we are done.

\end{proof}

\begin{secProp}
Let $H$ be a Hilbert space of dimension $d$ and let $A=(A_1,...,A_n)$
be a commuting tuple of operators acting on $H$.
Then $\sigma_T(A)= \sigma_{\textbf{A}}(A)$.

\end{secProp}

\begin{proof}
For $d=1$ we get $A_i=\lbrace \alpha_i \rbrace$. 
One can easily see that for $\alpha=(\alpha_1,...,\alpha_n)$ 
the Koszul complex of $A-\alpha$ is the zero complex, thus $\alpha \in \sigma_T(A)$.
By
the preceding together with \ref{ABC} we get $\sigma_T(A)\subset \sigma_{\textbf{A}}(A)=\lbrace \alpha \rbrace$.
 We conclude that $\sigma_T(A)=\sigma_{\textbf{A}}(A)$.
 
Now assume this holds for $d<j$ and let $A=(A_1,...,A_n)$ be commuting 
operators acting on a Hilbert space $H$ of dimension $j$.
Let $e_1,...,e_j$ be some orthonormal basis for which the matrix representation of each $A_i$ is upper triangular, note that $A_ke_1=\lambda_1^{(k)}$
thus $e_1$ is a common eigenvector of $A_1,...,A_n$.
Let $V=Span\lbrace e_1 \rbrace$ and denote by $\sigma_T(A,V)$ and
$\sigma_T(A,{H}/{V})$ the Taylor spectrum of the operators induced
by $A$ on $V$ and ${H}/{V}$ respectably. 
By lemma 1.2 of \cite{taylor1970joint} we get that $\sigma_T(A,H/V) \subset 
\sigma_T(A) \bigcup \sigma_T(A,V)$.
We denote by $C_k$ the $(j-1)\times(j-1)$ matrix that is obtained by removing the first row and column from the representation matrix of $A_k$ (with respect to $e_1,...,e_j$). Note that $C_k$ is the representing matrix of the operator induced by $A_k$ on the quotient space and that 
$C=(C_1,...,C_n)$ is a commuting tuple of upper triangular matrices.
By induction we 
get $\sigma_T(C)=\lbrace (\lambda_k^{(1)},...,\lambda_k^{(n)} )\vert k=2,...,d\rbrace $.
Now since $T_k\upharpoonright_ V=\lambda_1^{(k)}$ we get by induction that
 $ \sigma_T(A,V)=\lbrace (\lambda_1^{(1)},...,\lambda_1^{(n)} )\rbrace$.
 So if we will show $\sigma_T(A,V)\subset \sigma_T(A)$
we will be done. 
Let $D_\lambda$ be the Koszul map associated to the Koszul complex of
 $A-(\lambda_1^{(1)},...,\lambda_1^{(n)} )$. Note that since $e_1$
 is a common eigenvector of $A$ we have that $e_1$ is in the intersection of the kernels of $A_k-\lambda_k$ and one can easily see that $D_\lambda$ restricted to
 $\Lambda^1(H)$ is not injective and thus the 
  Koszul complex of $A-(\lambda_1^{(1)},...,\lambda_1^{(m)} )$ is not exact and we are done.
\end{proof}
\begin{secRemark}
Since  $\sigma_T(A)= \sigma_{\textbf{A}}(A)$ it follows from previews discussion that for any commutative algebra $\textbf{B}$ 
containing $\textbf{A}$ we have that $\sigma_T(A)= \sigma_{\textbf{B}}(A).$ 
\end{secRemark}
There is another type of joint spectrum we have not mentioned yet which was introduced  Waelbroeck. 
\begin{secDef}
Let $A=(A_1,...,A_n)$ be a tuple of commuting bounded operators acting on a Hilbert space $H$. The \textbf{Waelbroeck joint spectrum} of $A$ which we will denote by 
$\sigma_W(A)$ is defined to be the set of all
$\lambda=(\lambda_1,...,\lambda_n)\in \mathbb{C}^n$ such that $q(\lambda)$ belongs to the spectrum of $\sigma (q(A))$ for every multivariate polynomial $q=q(z_1,...,z_n)$.
\end{secDef}
\begin{secRemark}
In \cite[Section 1.1]{arveson1972subalgebras} it is shown that the Waelbroeck joint spectrum is the algebraic joint spectrum of a tuple of commuting operators with respect to the smallest inverse closed, unital, commutative Banach algebra generated by these operators.   
By the previous remark we have that for matrices the Waelbroeck joint spectrum coincides with  both of the joint spectra we have discussed. 
\end{secRemark}
 
\subsection{The normal case}

\begin{secProp}
Let $N=(N_1,...,N_n)$ be an $n$-tuple of commuting normal operators acting 
on a Hilbert space $H$ and let $C^*(N)$ be the unital $C^*$-algebra 
generated by $N$. Then $\sigma_T(N)=\sigma_{C^*(N)} (N)$.

\end{secProp}
\begin{proof}
Let $\lambda=(\lambda_1,...,\lambda_n)\in \mathbb{C}^n$.
By \cite[Corollary 3.9]{curto1981fredholm} for it follows that the Koszul complex of $N-\lambda$ is exact if and 
only if  $\Sigma_k(N_k-\lambda_k) (N_k-\lambda_k)^*$ is invertible in $C^*(N)$, so we only need to show that 
$\lambda\in \sigma_{C^*(N)}(N)$ if and only if $\Sigma_k (N_k-\lambda_k)(N_k-\lambda_k)^*$ is singular.

So if $\lambda=(\lambda_1,...,\lambda_n)\in \sigma_{C^*(N)}(N)$ 
define $f(z_1,...,z_n)=\vert z_1-\lambda_1 \vert ^2 +... \vert z_n-\lambda_n\vert ^2$ then by the spectral mapping theorem (for several commuting normal operators) $f(N)=\Sigma_k (N_k-\lambda_k)(N_k-\lambda_k)^*$
and $\sigma(\Sigma_k (N_k-\lambda_k)(N_k-\lambda_k)^*)=f(\sigma_{C^*(N)})$
so $0=f(\lambda)\in \sigma(\Sigma_k (N_k-\lambda_k)(N_k-\lambda_k)^*)$ 
thus it is singular.

The other direction follows from Theorem \ref{ABC}.
\end{proof}
\begin{secRemark}
Let us again denote by $\textbf{A}$ the unital Banach algebra generated be $N$.
Then it is not true that $\sigma_T(N)=\sigma_{\textbf{A}}(N)$.
For example one can consider the bilateral shift $S$ acting on $l^2(\mathbb{Z})$.
It can be shown that $\sigma_{\overline{\Alg(S)}}(S)=\overline{\mathbb{D}}$ while, since $S$ is a unitary, $\sigma_{C^*(S)}(S)=\sigma_T(S)\subset \mathbb{T}$. 
\end{secRemark}

\chapter{Polynomial normal $X$-dilations}

\begin{chapDef}
\label{PD}
Let $X$ be some compact subset on the complex plane and let $T$ be an operator in $B(H)$. A  \textbf{polynomial normal $X$-dilation} of $T$ will consist of of a normal
operator $N$ acting on some Hilbert space $K$ containing $H$ such that
the spectrum of $N$ is contained in $X$ and such that
$$ B^k=P_HN^k \upharpoonright _H \ ,\ k\in \mathbb{N}.$$  
\end{chapDef}
\begin{chapRemark}
One can clearly see that a unitary dilation of an operator is simply 
a polynomial $\mathbb{T} $-dilation.
\end{chapRemark}
We are now able to define the multi-variable case.
As we shall concern ourselves only with the finite dimensional case, $H$ will always be a $d$-dimensional Hilbert
space, $T=(T_1,...T_n)$ shall be a commuting tuple of operators acting on $H$, and $X$ will be a compact subset of $\mathbb{C}^n$. We shall also denote the joint spectrum of $T$ by $ \sigma(T)$ (which is the same for both kinds of joint spectra as was seen before).

We now introduce the following definitions: 
\begin{chapDef}
Let $X$ be a subset of the complex plane and let $T=(T_1,...T_n)$ be a tuple of operators in $B(H)$.
We shall say $T$ has a \textbf{polynomial normal $X$-dilation} 
if there exists an $n$ tuple of commuting normal
operators $N=(N_1,...,N_n)$ acting on some Hilbert space $K$ containing $H$ such that
$\sigma(N)$ is contained in $X$ and such that
$$ T_1^{m_1}...T_n^{m_n}=P_HN_1^{m_1}...N_n^{m_n} | _H \ ,\ m_1,...,m_n\in \mathbb{N}.$$  
 
\end{chapDef}
We shall have a concept of spectral sets.
\begin{chapDef}
A compact subset $X$ of $\mathbb{C}^n$ will be called a \textbf{polynomial spectral set} for a tuple  $T=(T_1,...T_n)$ 
if for any polynomial $q\in \mathbb{C}[z_1,...,z_n]$ one has
$$\Vert q(T_1,...,T_n) \Vert \leq \Vert q \Vert_{\infty,X}$$
$X$ will be called a \textbf{complete polynomial spectral set} for $T$ if for any $l\in \mathbb{N}$ and any $l\times l$ matrix $Q$ with 
entries in $ \mathbb{C}[z_1,...,z_n]$, one has that $$\Vert (Q_{i,j}(T) )\Vert \leq \Vert (Q_{i,j})\Vert,$$
where the norm on the left side of the inequality is the operator norm on the direct sum  $\oplus_{i=1} ^l H$
and the norm on the right is $\sup_{z\in X}\Vert (Q_{i,j}(z_1,...,z_n))\Vert_{M_l(\mathbb{C})}.$
\end{chapDef}

A direct consequence of \cite[Theorem 1.2.2]{arveson1972subalgebras} yields us the following connection.

 \begin{chapTheorem}
 \label{T3}
Let $X\in \mathbb{C}^n $ and let $T=(T_1,...,T_n)$ be a tuple of commuting operators on a Hilbert space $H$. Then
$T$ has a polynomial normal $X$-dilation if and only if $X$ is a complete polynomial spectral set for $T$.
 \end{chapTheorem}
 
 The following definition shall help us introduce a connection between spectral sets and polynomial spectral sets.
   
\begin{chapDef}
Let $X\subset \mathbb{C}^n$ be compact. The \textbf{polynomially convex hull} of $X$, denoted by $\widehat{X}$, is defined as 
$$\widehat{X}=\{ z\in \mathbb{C}^n : |p(z)|\leq \max_{ \xi\in X} |p(\xi )| \text{ for all multivariate polynomials} \} .$$
$X$ will be called \textbf{polynomially convex} if $X=\widehat{X}$.
\end{chapDef} 

\begin{Ex}
\label{Q1}
Any finite set $X$ is polynomially convex.

\end{Ex}
\begin{proof}
Let $X=\{ w_1,...,w_k\}\subset \mathbb{C}^n$ where $w_i=(w_i^{(1)},...,w_i^{(n)})$.
It is easy to see that $X\subset \widehat{X}$ as for the other direction assume otherwise, then there exist a point $w_0 \in \widehat{X}\setminus X$. For each $i\in\{1,...,k \} $ there exist a $j_i$ such that $w_0^{(j_i)}\neq w_i^{(j_i)}$. Now consider the polynomial 
$$p(z_1,...,z_n)=\prod_i (z_{j_i}-w_i^{(j_i)}).$$
Note that $p(w_0)\neq 0 $ and that for all $i\in \{1,...,k\}$, $p(w_i)=0$ and therefore 
$$|p(w_0)|> \max_{\xi \in X} |p(\xi)|$$ thus $w_o$ is not in $\widehat{X}$, a contradiction.

\end{proof}
\begin{chapRemark}
\label{Q2}
The polynomially convex hull of a compact set is compact.
\end{chapRemark}
\begin{proof}
For each multivariate polynomial $p$ define $$F_p=\{ z\in \mathbb{C}^n : |p(z)|\leq \max_{ \xi\in X} |p(\xi )| \}.$$
We notice that $$\widehat{X}=\bigcap_{p}F_p.$$ Since each $F_p$ is a closed set so is $\widehat{X}$.
In order to see $\widehat{X}$ is bonded consider the polynomials $$q_i=(z_1,...,z_n)=z_i,\ i\in\{1,...,n\}.$$ Let $$r_i=\max_{\xi\in X} |q_i(\xi)|$$
and set $$r= \sqrt{r_1^2+...+r_n^2}.$$ Then $$\widehat{X}\subset \bigcap_{i=1}^n F_{q_i} \subset B(r)$$
 thus $\widehat{X}$ is compact.
\end{proof}
\begin{chapRemark}
\label{Q3}
For any $z \notin \widehat{X}$ there exist a polynomial $p_z$ such that
$$|p_z(z)|>\max_{ \xi\in X} |p_z(\xi )|$$
\end{chapRemark}
 \begin{chapProp}
 \label{Q4}
Let $T=(T_1,...,T_n)$ be a tuple of commuting bounded operators in $B(H)$.
Let $X\subset \mathbb{C}^n$ be a polynomial spectral set for $T$. Then 
$\sigma(T) \subset \widehat{X}$.
 \end{chapProp}
\begin{proof}
Assume otherwise, then there exist $\lambda \in \sigma(T) \smallsetminus \widehat{X}$. By \ref{Q3} we have some polynomial $p_{\lambda}$ such that 
$|p_{\lambda}(\lambda)|>\max_{\xi\in X}|p_{\lambda}(\xi)|$.  Since $\lambda\in \sigma(T)$
we have that $p_{\lambda}(\lambda)\in \sigma (p_{\lambda}(T))$ and thus
$$\Vert p_{\lambda}(T) \Vert \geq | p_{\lambda}(\lambda) | > \max_{\xi\in X}|p_{\lambda}(\xi)| .$$
This is a contradiction for we assumed $X$ is a polynomial spectral set for $T$.
\end{proof}

 \chapter{Normal $X-m-$dilations}
In this chapter we shall focus on dilations of operators acting on finite dimensional 
spaces. We will introduce a different kind of dilation that may have
some advantage in this particular setting and give connections to polynomial normal $\partial X$ dilations. For convenience we shall refer to polynomial normal $X$-dilations simply as normal $X$ dilations.
We begin this discussion with the following well known observation.

\begin{chapLemma}
Let $B$ in $M_d(\mathbb{C})$ and let $X\subset \mathbb{C}$ be a polynomial spectral set for $B$.
If $X$ is finite then $B$ is normal. 
\end{chapLemma}

\begin{proof}
We first note that since $X$ is a finite set
we have that $X=\widehat{X}$ and since $X$ is a
 polynomial spectral set for $B$ we get from proposition \ref{Q4} that $\sigma(B) \subset X$.
We now turn to show that $B$ is diagonalizable. Indeed, if it not diagonalizable then there exists a $\lambda \in \sigma(B)$
such that $(z- \lambda )^2$ divides $p_m(z)$ - the minimal polynomial of $B$.
Define $$q(z)=\prod_{x\in X} (z-x)$$ Then $q(z)$ is not in the ideal generated by $p_m(z)$ and we have
$$0< \Vert q(B) \Vert \leq \Vert q \Vert _{\infty , X} = 0,$$
which is a contradiction.
 
Now assume $X \subset \{0,1\} $. Then $\sigma(B)\subset \{0,1\}$,
 we shall show $B$ is a projection.
In the case $\sigma(B)$ consists of one point, since $B$ is diagonalizable, we have that $B$ is either the zero matrix or the identity.
Thus we assume $X=\sigma(B)= \{ 0,1 \} $. 

Let $v_0,v_1$ be two eigenvectors of eigenvalues $0$ and $1$, respectively. 
It will suffice to show that $\langle v_0,v_1 \rangle=0$.
Assume this is not the case and note that we can assume $\langle v_0,v_1 \rangle<0$ (by replacing $v_0$ with $\alpha v_0$
for an appropriate $\alpha$). Moreover, by scaling down $v_0$ we can arrange $\Vert v_0 \Vert ^2 +2\langle v_0,v_1 \rangle$
is strictly negative and we have:
$$\Vert v_0+v_1 \Vert ^2 = \Vert v_0\Vert ^2 +2\langle v_0,v_1 \rangle  + \Vert v_1 \Vert ^2 < \Vert v_1 \Vert ^2 = 
\Vert B(v_0 + v_1) \Vert ^2$$
Thus $\Vert B \Vert > 1$, but $\{ 0,1 \}$ is a spectral set for $B$ so $\Vert B \Vert \leq \max_{\{0,1 \}} \vert z \vert =1$
, a contradiction.

As for the general case let $B \in M_d(\mathbb{C}) $ with $X$ as its spectral set and let 
$\sigma(B)=\{\lambda_1,...,\lambda_k \}$,
and recall that $\sigma(B)\subset X$ . For each
$\lambda_j \in \sigma(B)$ we define the polynomial
$$p_j(z)=\prod_{x\in X\setminus \{\lambda_j \} } \dfrac{(z-x)}{(\lambda_j - x)}$$
Then $p_j(x)=\delta_{x,\lambda_j}$ on $X$. Note that $p_j(X) \subset \{0,1\}$ is a spectral set for $p_j(B)$, thus by the
preceding $p_j(B)$ is an orthogonal projection.
In addition, we have that $p_i(z)p_j(z) \equiv 0$ on $\sigma(B)$ for $i \neq j$, therefore $p_i(B)p_j(B)$ must be the zero
matrix. To conclude we note that $B=\sum_i \lambda_ip_i(B)$, a weighted 
sum of orthogonal projections, therefore $B$ is normal.

\end{proof}

\begin{chapCor}
If $B=(B_1,...,B_n)$ is a commuting tuple of operators acting on a finite dimensional Hilbert space $H$ which has a normal $X$ dilation $N=(N_1,...,N_n)$ acting on some finite dimensional Hilbert space $K$. Then $B_j$ is normal for $j=1,...,n$.
\end{chapCor}
\begin{proof}
We have that $N_j$ is a normal $ \sigma(N)$ dilation for $B_j$ thus $\sigma(N)$ is a
polynomial spectral set for $B_j$.
Since $N_j$ acts on a finite dimensional space, $\sigma(N_j)$ finite, thus in view of the previous lemma we obtain that $B_j$ is normal.  
\end{proof}

In view of the last corollary it turns out that in order to try to understand a tuple of commuting matrices through their dilation
one needs to invoke operator theory on infinite dimensional spaces.
In order to try to stay within the realm of finite dimensional linear algebra we  introduce the following definition of a dilation.

\begin{chapDef}
Let $X\subset \mathbb{C}^n$ be compact and let $m\in \mathbb{N}$.
Let $B=(B_1,..,B_n)$ be a commuting tuple of operators acting on a Hilbert space $H$.
A \textbf{normal $ X-m-$dilation} for $B$ is an 
$n$ tuple of commuting normal operators $N=(N_1,...,N_n)$ acting on a Hilbert space $K$ containing $H$ such that
 $\sigma(N)\subset  X$ and such that $$B_1^{m_1}...B_n^{m_n}=P_H N_1^{m_1}...N_n^{m_n} \upharpoonright _H$$
 for all non-negative integers $m_1,...,m_n$ satisfying $m_1+...,+m_n \leq m$.   

\end{chapDef}

\begin{chapRemark}
When taking $X=\mathbb{T}^n$ we have that a normal $X-m-$dilation is simply 
 a unitary $m$-dilation which was presented in the introduction.

\end{chapRemark}
We now state our main result.
\begin{chapTheorem}
\label{AAS}
Let $T=(T_1,...,T_n)$ be a tuple of commuting operators acting of a Hilbert space $H$ of dimension $d$.
Then the following are equivalent
\begin{enumerate}
\item
$T$ has a normal $X$ dilation. 
\item
For any positive integer $m$,
$T$ has a normal $X-m-$dilation acting on a finite dimensional space whose dimension depends on $d$.
\end{enumerate}
\end{chapTheorem}
Before we proceed let us introduce some notation.

\begin{chapDef}
Let $X$ be a compact Hausdorff space and let $\mathcal{B}$ be the Borel
$\sigma$-algebra of $X$.
A positive operator valued measure (POVM) on $X$ is a map $\mu : \mathcal{B} \rightarrow B(H)$
that satisfies the following:
\begin{enumerate}
\item

For every countable collection of disjoint Borel sets 
$\lbrace B_i \rbrace_{i\in \mathbb{N}}$ with union $B$ we have
\begin{center}
$\langle \mu (B)x,y \rangle =\sum_i \langle \mu (B_i)x,y \rangle $, 
for all $x,y\in H.$
\end{center}
\item
$$\sup \lbrace \Vert \mu(B) \Vert : B\in \mathcal{B} \rbrace < \infty.$$
\item 
For all $x,y\in H$ we have that the complex measure given by $\mu_{x,y}(B)=\langle \mu(B)x,y\rangle$ is regular.
\item 

$\mu (B)$ is a positive operator for all $B\in \mathcal{B}$. 

\end{enumerate}
\end{chapDef}
Given a POVM, one obtains a bounded, linear map
$$\phi_{\mu}: C(X) \rightarrow B(H) $$ by 
$$\langle\phi_{\mu}(f)x,y\rangle= \int f d \mu_{x,y}.$$ 
Note that by condition 4 it follows that $\phi_{\mu}$ is a positive map.
Conversely, given a bounded positive map $\phi :C(X) \rightarrow B(H)$, then if we define the regular Borel measures $\{ \mu_{x,y} \}$ for each $x$ and $y$ in $H$ 
by the formula above, then for each Borel set $B$, there exists a unique, bounded positive operator $\mu(B)$, defined by $$\langle\phi_{\mu}(f)x,y\rangle= \int f d \mu_{x,y},$$  and the map $B \rightarrow \mu(B)$ is a POVM. Thus we obtain a one-to-one correspondence between the positive maps from $C(X)$ into $B(H)$ and POVM
(see \cite[Chapter 4, Ex. 4.10]{paulsen2002completely} for more details).
\begin{chapTheorem}
 \label{W1}
Let $\mu$ be a positive semi-definite matrix valued measure on $M_d(\mathbb{C})$ with a compact support $K\subset \mathbb{C}^n$ such that $\mu (K)$ is the identity and let $f_1,f_2,...f_k\in C(K)$. Then there exists  $M\in \mathbb{N}$, $w_1,w_2,...w_M$ in $K$ and positive semi-definite matrices $A_1,A_2,...,A_M$ in $M_d (\mathbb{C})$ such that
$\Sigma_jA_j=I_d$ and such that for any $f\in span\{f_1,...,f_k\}$ $$\int f d\mu=A_1f(w_1)+A_2f(w_2)+...+A_Mf(w_M)$$ holds. 
\end{chapTheorem}
\begin{chapRemark}
This theorem can be viewed as a POVM analogue of what is known as Tchakaloff's theorem (the original statement is set for positive, compactly supported measure which is absolutely continuous with respect to Lebesgue $n$-volume measure).
We should also stress that the proof takes after Putinar's proof of Theorem 1 in \cite{ putinar1997note }.

\end{chapRemark}

\begin{proof}
We first assume  $f_1,...,f_k$ are all real valued functions.  
 Define the following $\mathbb{R}$-linear transformation  $T:M_d(\mathbb{C})\rightarrow \mathbb{R}^{2d^2}$
$$(a_{i,j}+ib_{i,j})_{i,j}\mapsto (a_{1,1},b_{1,1},a_{1,2},...,a_{d,d},b_{d,d}),$$ and set $L=k+1$.
We define the map $v:M_d(\mathbb{C}) \times K\rightarrow \mathbb{R}^{2d^2L}$ to be as such:
$$v(A,x)=(T(A), f_1(x)T(A),f_2(x)T(A),...,f_k(x)T(A)).$$ 

Let $C= \lbrace v(A,x), A\in M^+_d(\mathbb{C}), x\in K \rbrace$.     Our aim is to show that
$\conv(C)$ is a closed set.
 We begin by showing that $C$ is closed.
Let $(v(A_i,x_i))_i$ be a sequence in $C$ which converges to some $w\in\mathbb{R}^{2d^2L}$. It follows $(T(A_i))_i$ is a bounded sequence in $\mathbb{R}^{2d^2}$. 
Since $T$ is an homeomorphism and 
$M^+_d(\mathbb{C})$ is closed we have that $(A_i)_i$ is contained in some compact subset of $M^+_d(\mathbb{C})$. Note that $K$ is compact and $v$ is a continuous map and therefore $(v(A_i,x_i))_i$ is contained in some compact subset of $C$ so we get that $w$ is in $C$, hence $C$ is closed.  

    We return to show $\conv(C)$ is closed.
Note that by Caratheodory's theorem (\cite{davidson2010real} , p. 453) every element in $\conv(C)$ can be written as a sum of at most $2d^2L+1$ elements in $C$ and since $0\in C$ we can assume 
this sum is exactly of $2d^2L+1$ such elements. Thus W.L.O.G.
let $(u_i)_i=(\Sigma_jv(A^{(i)}_j,x^{(i)}_j))_i$  be a sequence in $\conv(C)$ which converges to some $u$.
 Then $(\Sigma_jT(A^{(i)}_j))_i=(T(\Sigma_jA^{(i)}_j))_i$
is a bounded sequence and therefore so is $(\Sigma_jA^{(i)}_j)_i$.
 Since $\Sigma_jA^{(i)}_j$  
is a positive semi-definite matrix, we have that for every fixed $i_0,j_0$
$  \Sigma_jA^{(i_0)}_j - A_{j_0}^{(i_0)} $ is also a positive semi-definite matrix and thus by Theorem 2.2.5 of \cite{mr} we have that:
$$\parallel A_{j_0}^{(i_0)} \parallel \leq \parallel \Sigma_jA^{(i_0)}_j \parallel$$  and so $(A_j^{(i)})_{i,j}$ is bounded. We conclude that for every $1 \leq j_0 \leq 2d^2L+1$, $(v({A^{(i)}_{j_0},x^{(i)}_{j_0}}))_i$
is a bounded sequence in $C$. 
Consequently, there is a convergent subsequence $$(\Sigma_jv(A^{(i)}_j,x^{(i)}_j))_i \rightarrow u= \Sigma_jv({A'_j,x'_j}) \in \conv(C)$$ and $\conv(C)$ is closed.

We now turn to show that the ``moment vector''
$$u=(T( \int{ 1d \mu }),T( \int{  f_1 d \mu}),...,T( \int{ f_k d \mu}))$$ is in $\conv(C)$.
We first recall that if $(g_i)_i$ is a bounded sequence measurable functions which converges uniformly  to a function $g$ then $\int g_i d\mu \rightarrow \int g d\mu$ in the weak operator topology which in our case (since $M_d(\mathbb{C})$ is of finite dimension) is equivalent to the norm topology.   
Now fixing $\varepsilon >0$ there are
 $B_1^{(\varepsilon)} ,...B_l^{(\varepsilon)} \subset K$ Borel sets and points
  $b_1^{(\varepsilon)} ,...,b_l^{(\varepsilon)} $, 
   where $b_i^{(\varepsilon)}\in B_i^{(\varepsilon)}$ 
 such that $\bigcup B_i^{(\varepsilon)}=K$ and for every $f_i$, $i=1,...,k$, one has $$\sup_{x \in K} \mid  f_i(x) - \Sigma _j f_i(b_j^{(\varepsilon)}) \chi _{B_j^{(\varepsilon)}} \mid <\varepsilon.$$ And so we have
$$v_{\varepsilon}= (\Sigma_jT(\mu( B_j^{(\varepsilon)} )),\Sigma _j  f_1(b_{j}^{(\varepsilon)}) T(\mu ( B_j^{(\varepsilon)} ) ),...,\Sigma _j  f_k((b_{j}^{ {(\varepsilon)}})) T(\mu ( B_j^{(\varepsilon)} ) ) =
$$ 
$$
(T(\int  \Sigma_j  \chi _{B_j^{(\varepsilon)}}d\mu)  , T(\int \Sigma _j  f_1(b_{j}^{(\varepsilon)} )\chi _{B_j^{(\varepsilon)}} d\mu) ,...,    T(\int \Sigma_j f_k((b_{j}^{(\varepsilon )}))\chi _{B_j^{(\varepsilon)}}d\mu))
\rightarrow_{\varepsilon \rightarrow 0} =  $$
 
$$(T( \int{ 1d \mu }),T( \int{  f_1 d \mu}),...,T( \int { f_k d \mu})).$$
 Since  $(v_\varepsilon) \subset \conv(C)$  and $\conv(C)$ is closed we conclude $u\in \conv(C)$. Thus by Caratheodory's theorem there are $M= 2d^2L+1$ points $w_1,w_2,...w_M$ in $K$ and $A_1,A_2,...,A_M$ positive semi-definite matrices such that\
$$u=  (T( \int{ 1d \mu }),T( \int{  f_1 d \mu}),...,T( \int { f_kd \mu}))=$$
$$\Sigma_jv(A_j,w_j)=(\Sigma_jA_j,\Sigma_jf_1(w_{j})A_j,...,\Sigma_jf_k(w_{j})A_j).$$

And so for each $i=1,...,k$ we have that
  $$T( \int{ f_id \mu })=\Sigma_j f_i(w_j)T(A_j)$$ and therefore $ \int{ f_id \mu }=\Sigma_j f_i(w_j)A_j$. Expanding linearly we obtain the wanted result for any $f$ in $Span\{f_1,...,f_k\}$.
  
For the general case Let $f_1,...,f_k$ be some functions in $C(K)$.
Then for $j=1,...,k$ define $g_{2j-1}(z)=Ref_j(z)$ and $g_{2j}(z)=Imf_j(z)$.
By applying the previous case to $g_1,...,g_{2k}$ and expanding linearly are able to finish the proof.
\end{proof}
Let us now return to the proof of theorem \ref{AAS}.

\begin{proof}
\textbf{(1) $\implies$ (2)}:
Assume $T=(T_1,...,T_n)$ has a normal $X$-dilation
$N=(N_1,...,N_n)$ acting on a space $K$ containing $H$.
Let $\nu$ be the spectral measure on $\sigma(N)$.
Then we have that for any $n$ variable polynomial $q$
$$\int q d \nu=q(N).$$
We define $\mu(\cdot )=P_H \nu(\cdot)|_H$.
It easy to verify that $\mu$ is a POVM taking values in $B(H)$ 
with a support $\sigma(N)$ and such that $\mu(\sigma(N))=I_H$.
Let  $\mathbb{C}_{\leq m}[z_1,z_2,...,z_n]$ denote the space of all polynomials in $n$ variables over $\mathbb{C}$ of degree at most $m$
and take some basis for $\mathbb{C}_{\leq m}[z_1,z_2,...,z_n]$. Then by Theorem \ref{W1} 
there exists 
 $ M \in \mathbb{N} $, $w_1,w_2,...,w_M$ in $K$,
 $w_j=(w_{1,j},...,w_{n,j})$ and positive operators $A_1,A_2,...,A_M$ in $B(H)$ 
 such that for any $q\in \mathbb{C}_{\leq m}[z_1,z_2,...,z_n]$ we have: 
$$ \int q d \mu =A_1q(w_1)+A_2q(w_2)+...+A_Mq(w_M). $$
$A_1,A_2,...,A_M$ can be thought of as a POVM on a set of $M$ points, and thus 
 by Naimark’s Theorem \cite[P. 40, Theorem 4.6]{paulsen2002completely} this POVM can be dilated to a spectral measure on this set. This simply means that there exists a Hilbert space $L$ containing $H$ and $E_1,...,E_M$ orthogonal projections on $L$ 
such that $\Sigma_jE_j=I$ and for any $j=1,...,M$, $A_j=P_HE_jP_H$.
Moreover $L$ can be chosen to be at most
$M \times d$ dimensional.

Let $i\in \lbrace1,...,n\rbrace$ and set $N_i=\Sigma_jw_{i,j}E_j$.
It is easy to see $N=(N_1,...,N_k)$ are commuting normal operators on $L$.
Note that by \ref{MM} $sp(N)$ is simply $\lbrace w_1,...,w_M\rbrace$ and thus 
 constitute a normal $X-m-$dilation.
 Given that $m$ was arbitrary this completes the proof.

 \textbf{(2) $\implies$ (1)}:
 Assume $T=(T_1,...,T_n)$  has  $X$-$m$-dilation for every $m$ and let $l$ be a positive integer.
 Let $Q=(q_{i,j})_{i,j=1}^l $ be a $l\times l$ matrix whose
entries are in $\mathbb{C}[z_1,...,z_n]$.
Set $m=\max_{i,j=1,...l}\deg(q_{i,j})$
and let $N_1,...,N_n$ a $ X$-$m$-dilation for $T_1,...,T_n$ acting on some finite dimensional space $K$, 
then we have
$$q_{i,j}(T_1,...,T_n)=P_Hq_{i,j}(N_1,...,N_n)|_H$$
for all $i,j=1,...,l$ and thus $$  \Vert ( q_{i,j}(T_1,...,T_n) ) \Vert \leq \Vert ( q_{i,j}(N_1,...,N_n) ) \Vert$$
where the norm is the operator norm on $H\oplus ... \oplus H $ on the left side and on $K\oplus ... \oplus K$ on the right side.
Moreover we have $$\Vert q_{i,j}(N_1,...,N_n) \Vert \leq sup\{ \Vert q_{i,j}(z) \Vert_{M_l} : z \in  X \}$$ combining both inequalities we get  
$$\Vert q_{i,j}(T_1,...,T_n) \Vert \leq sup \{ \Vert q_{i,j}(z) \Vert_{M_l} : z\in  X \}$$ Since $l$ was arbitrary we have that $X$ is a polynomial spectral set for
$T$. Then by Theorem \ref{T3} $T$ has a normal $X$-dilation.

 \end{proof}






 
\chapter{Directions for further research}
During our work on this note some questions have arisen, some are directly connected to 
the work we have done, some are not of an obvious context.
Here are some of them.

\begin{Qn}
Given two commuting matrices which are both contractions. Can we give an explicit construction of their unitary $m$-dilation? 

\end{Qn}
 We know such a dilation exists due to Ando's dilation Theorem and \ref{AAS}, but even in the case of $2\times 2$ matrices we have not been able to find such a construction.
 Finding it will able us to give a finite dimensional proof to Ando's inequality. Such a proof, to our knowledge, has not been found.

\begin{Qn}
If $T$ has $X$ as a complete spectral set can we give a construction  
to the $m$-normal dilations of $T$? Perhaps only for some specific class of matrices.
As of now we are not able to provide new examples (see \cite[Theorem 4.3]{levy2010dilation})
\end{Qn}
\begin{Qn}
Let $T\in B(H)$ be of norm strictly less then 1, and $\varphi_1,...,\varphi_n\in A(\mathbb{D})$ such that $\vert \varphi_j(z) \vert_{\infty, \mathbb{D}} \leq 1$ for $i=1,...,n$.
We notice that the operators $\varphi_1(T),...,\varphi_n(T)$ are a tuple of commuting contractions. Moreover it is easy to see that the polydisk $\mathbb{D}^n$ is a complete spectral set for $\varphi_1(T),...,\varphi_n(T)$, thus a unitary dilation
exists.
Can we characterize tuples of commuting contractions that arise this way?

\end{Qn}
It is known that all tuples of commuting $2\times 2$ contractions can be represented in such way \cite{holbrook1992inequalities}. 
We were also able to show this for the following type of matrices.
\begin{chapDef}
A $d \times d$ matrix $A=(a_{i,j})$ will be called lower triangular Toeplitz matrix
if it is of the form
$$A=\begin{bmatrix}
a_0&0&...&...&...&0\\ a_1& a_0&0&0&...&0
\\ a_2&a_1& a_0 &0&...&0
\\ \vdots & \ddots & \ddots & \ddots & \ddots &\vdots
\\ \vdots & \ddots & \ddots & \ddots & a_0 &0
\\ a_{d-1} & a_{d-2} & ... & a_2 & a_1 & a_0
\end{bmatrix}  $$ for some $a_0,...,a_{d-1}\in \mathbb{C} $
\end{chapDef}

\begin{chapRemark}
Notice that if we set $p_A(z)=a_0+a_1z+...+a_{d-1}z^{d-1}$ then
$p_A(S_d)=A$ where $S_d$ is the nilpotent shift, i.e. $$S_d=\begin{bmatrix}
0&...&...&...&...&0\\ 1& 0&0&0&...&0
\\ 0&1& 0 &0&...&0
\\ \vdots & \ddots & \ddots & \ddots & \ddots &\vdots
\\ \vdots & \ddots & \ddots & 1 & 0 &0
\\ 0 & ... & ... & 0 & 1 & 0
\end{bmatrix}.  $$ It is not difficult to see that $S_d$ is a contraction.
\end{chapRemark}
\begin{chapProp}
Let $A_1,...,A_n\in M_d(\mathbb{C})$ be contractive lower triangular Toeplitz matrices. Then there exists $\varphi_1,...\varphi_n\in A(\mathbb{D})$
such that for any $j=1,...,k$ $\max_{z\in \mathbb{D}} | \varphi_j(z)| \leq 1$
and $\varphi_j( S_d)=A_j$.
\end{chapProp}
\begin{proof}
By the previous remark we have $p_{A_1},...,p_{A_n}$ in $\mathbb{C}[z]$ such that $p_{A_j}(S_d)=A_j$. 
We recall the Caratheodory-Schur interpolation problem \cite[Proposition XXVII.7.2]{gohberg2003basic} which asserts that a lower triangular Toeplitz matrix $A$ is a contraction if and only if there exists a $\varphi \in A(\mathbb{D})$
such that $\widehat{\varphi}(k)=a_k$, $k=0,...d-1$ (where $\widehat{\varphi}(k)$ is the $k$-th Fourier coefficient of $\varphi$) and such that $\sup_{z\in \mathbb{D}} |\varphi (z) | \leq 1$. 
Note that since $S_d^n=0$ for any $n>d-1$ we have that $\varphi_j(S_d)=p_{A_j}(S_d)=A_j$ and we are done.
\end{proof}

\bibliographystyle{plain}
\bibliography{myref}
\end{document}

%% file: thesisdef.tex
\usepackage{amsmath}
\usepackage{amssymb}
\usepackage{amsthm}
\usepackage{cite}
\usepackage[all,cmtip]{xy}

\newcommand{\conv}{\operatorname{conv}}
\newcommand{\Alg}{\operatorname{Alg}}
\newcommand{\Ker}{\operatorname{Ker}}
\newcommand{\Ran}{\operatorname{Ran}}

\newcommand{\restr}[2]{{
  \left.\kern-\nulldelimiterspace 
  #1 
  \vphantom{\big|} 
  \right|_{#2} 
  }}